\pdfoutput=1
\RequirePackage{ifpdf}
\ifpdf 
\documentclass[pdftex]{sigma}
\else
\documentclass{sigma}
\fi

\numberwithin{equation}{section}

\newtheorem{Theorem}{Theorem}[section]
\newtheorem*{Theorem*}{Theorem}
\newtheorem{Corollary}[Theorem]{Corollary}
\newtheorem{Lemma}[Theorem]{Lemma}

 { \theoremstyle{definition}

\newtheorem{Example}[Theorem]{Example}
\newtheorem{Remark}[Theorem]{Remark} }

\def\C{{\mathbb C}}
\def\la{\lambda}
\def\al{\al}
\def\dl{\delta}
\def\der{\partial}
\def\Z{\mathbb Z}
\def\FF{\mathbb F}

\def\slt{\mathfrak{sl}_2}

\def\ox{\otimes}
\def\Om{\Omega}
\def\ka{\kappa}
\def\Ik{{\mathcal{I}_k}}
\def\Sym{{\text{Sym}}}
\def\Alt{{\text{Alt}}}

\begin{document}
\allowdisplaybreaks

\newcommand{\arXivNumber}{2304.07843}

\renewcommand{\PaperNumber}{061}

\FirstPageHeading

\ShortArticleName{Polynomial Solutions Modulo $p^s$ of Differential KZ and Dynamical Equations}

\ArticleName{Polynomial Solutions Modulo $\boldsymbol{p^s}$ of Differential KZ\\ and Dynamical Equations}

\Author{Pavel ETINGOF~$^{\rm a}$ and Alexander VARCHENKO~$^{\rm b}$}

\AuthorNameForHeading{P.~Etingof and A.~Varchenko}

\Address{$^{\rm a)}$~Department of Mathematics, MIT, Cambridge, MA 02139, USA}
\EmailD{\href{mailto:etingof@math.mit.edu}{etingof@math.mit.edu}}

\Address{$^{\rm b)}$~Department of Mathematics, University of North Carolina at Chapel Hill,\\
\hphantom{$^{\rm b)}$}~Chapel Hill, NC 27599-3250, USA}
\EmailD{\href{mailto:anv@email.unc.edu}{anv@email.unc.edu}}

\ArticleDates{Received April 18, 2023, in final form August 23, 2023; Published online September 01, 2023}

\Abstract{We construct polynomial solutions modulo $p^s$ of the differential KZ and dynamical equations where $p$ is an odd prime number.}


\Classification{81R12; 11C08; 14H52}

\section{Introduction}\label{sec 1}

The KZ equations were introduced by Knizhnik and Zamolodchikov \cite{KZ}
to describe the differential equations for conformal blocks on the Riemann sphere.
Different versions of the KZ equations appear in mathematical physics, algebraic geometry
and the theory of special functions, see, for example, \cite{EFK, MO}.
One of the important properties of the KZ equations is their
realization
 as suitable Gauss--Manin connections. This construction gives a presentation
 of solutions of the~KZ equations by multidimensional hypergeometric integrals,
see \cite {CF,DJMM, SV1}.

The fact that certain integrals of closed differential forms over cycles
satisfy a linear differential equation
 follows by Stokes' theorem from a suitable cohomological relation,
in which the result of the application of the corresponding differential operator to the integrand of an integral
equals the differential of a differential form of one degree lower.
Such cohomological relations for the~KZ equations associated with arbitrary
Kac--Moody algebras were developed in \cite{SV2}.

The KZ equations possess a bispectrality property -- they have a compatible system of dynamical equations
with respect to associated dynamical parameters, see \cite{EV, FMTV, MO, OS, TV2}.

Let $p$ be an odd prime.
In \cite{SV3,V3}, the differential KZ equations were considered modulo $p^s$,
and polynomial solutions modulo $p^s$
were constructed as analogs of the hypergeometric integrals.
The construction was based on the fact
that all cohomological relations described in \cite{SV2} are defined over $\mathbb Z$ and can be reduced
modulo $p^s$.
Studying solutions modulo of $p^s$ sheds light on solutions of the KZ equations both over
the field of complex numbers and over $p$-adic fields, for example, see \cite{SmV}.

In this paper, we consider the joint system of the differential KZ and differential dynamical
equations, the system introduced in \cite{FMTV}, and construct polynomial solutions modulo $p^s$ of the joint system
as analogs of the
corresponding hypergeometric integrals with an exponential term.
For this purpose, one needs to represent
the exponential function ${\rm e}^{\lambda t}$ with an integer parameter~$\lambda$ by a polynomial in $t$ modulo $p^s$.
This can be done after replacing $\lambda$ with $p\lambda$.

 An interesting problem is to study the $p$-adic limit of the constructed polynomial solutions modulo $p^s$
 as $s\to\infty$, see examples of this limit for the differential KZ equations in \cite{V4,V3,VZ}.

The joint system of the KZ and dynamical equations has many versions: differential KZ equations and differential
dynamical equations,
differential KZ equations and difference dynamical equations, difference KZ equations and differential dynamical
equations, difference KZ equations and difference dynamical equations,
see, for example, \cite{EV,MO,OS,TV2}. The polynomial solutions modulo $p^s$ are constructed in this paper only for
the original joint system of the differential KZ and differential dynamical equations, although there
are examples of polynomial solutions modulo~$p^s$ in other cases, see \cite{MuV,RV,V4} and also Appendix \ref{sec app}.

In the remainder of the introduction we consider an example.

\subsection[Solutions over C]{Solutions over $\boldsymbol{\C}$}
\label{CaseC}

Consider the complex master function
\begin{gather*}
\Phi(t,z,\la) = {\rm e}^{\la t}\prod_{i=1}^{2g+1}(t-z_i)^{-1/2}
\end{gather*}
and the tuple of integrals
\begin{gather}
\label{int rep}
I(z,\la) = (I_1(z,\la), \dots,I_{2g+1}(z,\la))= \int_\dl\Phi(t,z,\la)
\left(\frac{1}{t-z_1}, \dots, \frac{1}{t-z_{2g+1}}\right){\rm d}t,
\end{gather}
where $\dl$ is a 1-cycle.

\begin{Theorem}\label{thm KZ}
The tuple $I(z,\la)$ satisfies the joint system of KZ and dynamical equations
\begin{gather}
\frac {\der I_j}{\der z_i} =
\frac 12 \frac{I_i-I_j}{z_i-z_j}, \qquad i\ne j,\label{KZ1}
\\
\frac {\der I_i}{\der z_i} = \la I_i - \frac 12 \sum_{j\ne i}\frac{I_i-I_j}{z_i-z_j},\label{KZ2}
\\
\frac {\der I_i}{\der \la} = z_i I_i + \frac 1{2\la} \sum_{j=1}^n I_j.\label{D}
\end{gather}
\end{Theorem}

The system of equations \eqref{KZ1} and \eqref{KZ2} is called the KZ equations of this example, the system of
equations \eqref{D} is called the dynamical equation. The solutions $I(z,\la)$ are called the hypergeometric solutions.

\begin{proof}
The proof uses the following identities:
\begin{align*}
 &\frac 1{(t-z_i)(t-z_j)}= \frac1{z_i-z_j}\bigg(\frac1{t-z_i} - \frac1{t-z_j}\bigg),
\\
&\frac{\der}{\der t}\Phi(t,z,\la)
=
\la \Phi(t,z,\la) -\frac 12 \Phi(t,z,\la) \sum_{j=1}^{2g+1} \frac 1{t-z_j},
\\
&\frac{\der}{\der t}\frac {\Phi(t,z,\la)}{t-z_i}
=
\Phi(t,z,\la) \bigg(\frac\la{t-z_i} -\frac 32\frac1{(t-z_i)^2} -\frac 12\sum_{j\ne i}
\frac{1}{(t-z_i)(t-z_j)}\bigg).
\end{align*}
For $j\ne i$, we have
\begin{gather*}
\frac {\der I_j}{\der z_i} = \int \Phi(t,z,\la) \frac {1/2}{(t-z_i)(t-z_j)}{\rm d}t
=\frac 12 (I_i-I_j).
\end{gather*}
This proves the first equation. Then
\begin{align*}
\frac {\der I_i}{\der z_i}
&= \int \Phi(t,z,\la) \frac {3/2}{(t-z_i)^2}{\rm d}t
\\
&= -\int\frac{\der}{\der t}\frac {\Phi(t,z,\la)}{t-z_i} {\rm d}t +
 \int \Phi(t,z,\la) \bigg(\frac\la{t-z_i} -\frac 12\sum_{j\ne i}
\frac{1}{(t-z_i)(t-z_j)}\bigg){\rm d}t
\\
&
= \la I_i - \frac 12 \sum_{j\ne i}\frac{I_i-I_j}{z_i-z_J}
\end{align*}
gives the second equation. We also have
\begin{align*}
\frac {\der I_i}{\der \la}
&= \int \Phi(t,z,\la)\frac{t-z_i+z_i}{t-z_i} {\rm d}t =
\int \Phi(t,z,\la) {\rm d}t + z_i\int \Phi(t,z,\la)\frac{1}{t-z_i} {\rm d}t
\\
\notag
&
=
\frac 1\la\int
\frac{\der}{\der t}\Phi(t,z,\la){\rm d}t
+\frac 1{2\la} \int \Phi(t,z,\la) \sum_{j=1}^{2g+1} \frac 1{t-z_j}{\rm d}t
+ z_i\int \Phi(t,z,\la)\frac{1}{t-z_i} {\rm d}t
\\
\notag
&
=z_i I_i + \frac 1{2\la} \sum_{j=1}^{2g+1} I_j.\tag*{\qed}
\end{align*}\renewcommand{\qed}{}
\end{proof}

The complex vector space of (multi-valued) solutions of the joint system of KZ and dynamical equations \eqref{KZ1}--\eqref{D}
is $(2g+1)$-dimensional.
Every solution of the joint system
has the integral presentation \eqref{int rep} for a suitable cycle $\dl$, see \cite[Theorem 6.1]{MTV} and an example in
\cite[Introduction]{MTV}.

\subsection{Exponential function}
We have
\begin{gather*}
{\rm e}^{\la t} = \sum_{m=0} ^\infty\la^{(m)}t^m
,\qquad\!
\la^{(m)}=\frac{\la^m}{m!},
\qquad\!\!
\binom{m+n}{m} \la^{(m+n)}
=\la^{(m)} \la^{(m)},
\qquad\! \frac{{\rm d}}{{\rm d}\la}\la^{(m)} = \la^{(m-1)}.
\end{gather*}
We set $ \lambda^{(m)}=0$ for $m<0$ by convention.

 Let $f(t,z) = \sum_{m=0}^\infty b_m(z) t^m$, $b_m(z)\in\Z_p[z]$, where
 $\Z_p$ is the ring of $p$-adic integers and
 $z=(z_1,\dots,z_{2g+1})$ are parameters.
Consider the decomposition
\begin{gather*}
{\rm e}^{\la t} f(t,z) =\sum_{k=0}^\infty c_k(z,\la) t^k,
\end{gather*}
where each $c_k(\la,z)$ is a linear function in finitely many symbols $\la^{(m)}$, $m=0, 1,\dots$,
whose coefficients lie in $\Z_p[z]$.

\begin{Lemma}\label{lem der}
 Let $s$, $\ell$ be positive integers.
Then the coefficient of $t^{\ell p^s-1}$ in the series $\frac{{\rm d}}{{\rm d}t} \!\big({\rm e}^{\la t} \! f(t,z)\big)$
 is divisible by $p^s$, that is, all coefficients of the corresponding linear function in
 symbols $\la^{(m)}$, $m\geq 0$,
 are divisible by $p^s$.
\end{Lemma}

\begin{proof}
It is enough to prove the lemma for $f(t)=t^a$. Then
\begin{align*}
\frac {\rm d}{{\rm d}t}{\rm e}^{\la t} t^a &=
\sum_{m=0}^\infty \big(\la \cdot \la^{(m)} t^{m+a} + a \la^{(m)}t^{m+a-1}\big)
= \sum_{m=0}^\infty \big((m+1) \la^{(m+1)} t^{m+a} + a \la^{(m)}t^{m+a-1}\big)
\\
&=
\sum_{k=0}^\infty \big((k-a+1)\la^{(k-a+1)}t^k + a\la^{(k-a+1)}t^k \big)
= \sum_{k=0}^\infty (k+1)\la^{(k-a+1)}t^k.\tag*{\qed}
\end{align*}\renewcommand{\qed}{}
\end{proof}

\begin{Lemma} If $\la\in \Z_p$, then $p^k\la^{(k)} \in \Z_p$ for all $k\geq 0$.
If $s$ is a positive integer and $k > s\frac{p-1}{p-2}$, then $p^k\la^{(k)}\in p^s\Z_p$.
\end{Lemma}

\begin{proof}
The maximal power of $p$ dividing $k!$ equals
$\big[\frac k{p}\big] + \big[\frac k{p^2}\big] + \big[\frac k{p^3}\big]+\cdots$
which is not greater than
$\frac k{p} + \frac k{p^2}+ \frac k{p^3}+ \dots =\frac k{p-1}.$
Hence the power of $p$ dividing $\frac {p^k}{k!}$ is not less than $k-\frac k{p-1} = k\frac{p-2}{p-1}>0$.
Hence $p^k\la^{(k)} \in \Z_p$. We have $ k\frac{p-2}{p-1}\geq s$ if $ k\geq s\frac{p-1}{p-2}$.
\end{proof}

Denote
\[
d(s) = \bigg[s\frac{p-1}{p-2}\bigg]+1,
\qquad
E_s(t) = \sum_{k=0}^{d(s)} \frac{t^k}{k!}.
\]

\begin{Corollary}
 If $\la\in \Z_p$, then ${\rm e}^{p\la t}\in \Z_p[[t]]$, $E_s(p\la t)\in \Z_p[t]$ and
\begin{gather*}
{\rm e}^{p\la t} \equiv E_s(p\la t) \pmod{p^s},
\\
\frac{\der}{\der t}E_s(p\la t)
\equiv
 p\la E_s(p\la t), \qquad
\frac{\der}{\der \la }E_s(p\la t) \equiv pt E_s(p\la t)
 \pmod{p^s},
 \\
 E_s(p\la (u+v)) \equiv E_s(p\la u) E_s(p\la v)
 \pmod{p^s}.
\end{gather*}

\end{Corollary}

\subsection[Remarks on p\^r]{Remarks on $\boldsymbol{p^r}$}

Let $v_p(a)$ denote the $p$-adic evaluation of $a$.

 Let
$r_1$, $r_2$ be relatively prime positive integers.
Denote $r=r_1/r_2$. Assume that $r>1/(p-1)$.
Then for a positive integer $k$, we have
\begin{gather*}
v_p\big(p^{k r}/k!\big) =\left(kr-\left(\left[\frac{k}{p}\right] + \left[\frac{k}{p^2}\right] + \cdots \right)\right)
> k\left(r-\frac1{p-1}\right)>0.
\end{gather*}
Hence, if $\la\in \Z_p\big[p^{1/r_2}\big]$, then $p^{kr}\la^{(k)} \in \Z_p\big[p^{1/r_2}\big]$.
Moreover, if $s$ is a positive integer and $k > s\frac{p-1}{r(p-1)-1}$, then $p^{kr}\la^{(k)}\in p^s\Z_p\big[p^{1/r_2}\big]$.

Denote
\[
d(r, s) = \left[s\frac{p-1}{r(p-1)-1}\right]+1,
\qquad
E_{r,s}(t) = \sum_{k=0}^{d(r,s)} \frac{t^k}{k!}.
\]
 If $\la\in \Z_p\big[p^{1/r_2}\big]$, then ${\rm e}^{p^{r}\la t}\in \Z_p\big[p^{1/r_2}\big][[t]]$, $E_{r,s}(p^{r}\la t)\in \Z_p\big[p^{1/r_2}\big][t]$ and
\begin{gather*}
{\rm e}^{p^{r}\la t} \equiv E_{r,s}(p^{r}\la t) \pmod{p^s},
\\
\frac{\der}{\der t}E_{r,s}(p^{r}\la t)
\equiv
 p^{r}\la E_{r,s}(p^{r}\la t),
 \\
\frac{\der}{\der \la }E_{r,s}(p^{r}\la t)
\equiv p^{r}t E_{r,s}(p^{r}\la t)
 \pmod{p^s},
 \\
 E_{r,s}(p^{r}\la (u+v)) \equiv E_{r,s}(p^{r}\la u) E_{r,s}(p^{r}\la v)
 \pmod{p^s}.
\end{gather*}

If $r=1/(p-1)$, then
\begin{gather*}
v_p\big(p^{k/(p-1)}/k!\big) =\left(\frac{k}{p-1}-\left(\left[\frac{k}{p}\right] + \left[\frac{k}{p^2}\right] + \cdots \right)\right)
> 0
\end{gather*}
and ${\rm e}^{p^{1/(p-1)} t} \in \Z_p\big[p^{1/(p-1)}\big][[t]]$ but $v_p\big(p^{k/(p-1)}/k!\big)$ does not grow as $k\to\infty$, and
so we get an infinite series rather than a polynomial.

\subsection{Reformulation of the equations}

Change the variable $\la\mapsto p\la$. Then the KZ and dynamical equations take the form
\begin{gather}
\label{KZ1p}
\frac {\der I_j}{\der z_i}
=
\frac 12 \frac{I_i-I_j}{z_i-z_j}, \qquad i\ne j,
\\
\label{KZ2p}
\frac {\der I_i}{\der z_i}
=
p\la I_i - \frac 12 \sum_{j\ne i}\frac{I_i-I_j}{z_i-z_j},
\\
\label{Dp}
\frac {\der I_i}{\der \la}
=
pz_i I_i + \frac 1{2\la} \sum_{j=1}^{2g+1} I_j.
\end{gather}
For any positive integer $s$, we construct below some vectors of polynomials
\begin{gather*}
I(z,\la) = (I_1(z,\la), \dots,I_{2g+1}(z,\la))
\end{gather*}
with coefficients in $\Z_p$ which satisfy the KZ equations \eqref{KZ1p}, \eqref{KZ2p} modulo $p^s$
if $\la\in \Z_p$ and satisfy the dynamical equations \eqref{Dp} modulo $p^{s}$ if
$\la\in \Z_p^\times$.

\subsection[Solutions modulo p\^s]{Solutions modulo $\boldsymbol{p^s}$}

For a positive integer $s$, define
\begin{gather*}
\Phi_s^o(t,z) = \prod_{i=1}^{2g+1}(t-z_i)^{(p^s-1)/2},
\qquad
\Phi_s(t,z,\la) = E_s(p\la t) \Phi^o_s(t,z) ,\\
\Psi^o_s(t,z)=
 \Phi_s^o(t,z) \bigg(\frac{1}{t-z_1}, \dots, \frac{1}{t-z_{2g+1}}\bigg),
\qquad
\Psi_s(t,z,\la) = E_s(p\la t)\Psi^o_s(t,z).
\end{gather*}
Consider the Taylor expansions
\begin{gather*}
 \Psi_s^o(t,z) = \sum_{m=0}^{(2g+1)(p^s-1)/2-1} c_m^o(z) t^m,
\qquad
\Psi_s(t,z,\la) = \sum_{d=0}^{(2g+1)(p^s-1)/2-1+d(s)} c_d(z, \la) t^d,
\end{gather*}
where each $ c_m^o(z)$ is a vector of polynomials in $z$ with integer coefficients,
and
\begin{gather*}
c_d(z,\la) = \sum_{m=0}^d p^{d-m}\la^{(d-m)} c_m^o(z).
\end{gather*}
For any positive integer $\ell$, denote
\begin{gather*}
I^{\ell}(z,\la) = c_{\ell p^s-1}(z,\la).
\end{gather*}
All coordinates of this vector are polynomials in $z, \la$ with coefficients in $\Z_p$.

\begin{Theorem}
\label{thm mod} Let $\ell$ be a positive integer. If $\la\in\Z_p$, then
 $I^{\ell}(z,\la)$ is a solution modulo $p^s$ of the KZ equations \eqref{KZ1p}, \eqref{KZ2p}.
If
$\la\in \Z_p^\times$, then $I^{\ell}(z,\la)$
 is a solution modulo $p^{s}$ of the dynamical equations \eqref{Dp}.

\end{Theorem}

We call such solutions the {\it $p^s$-hypergeometric solutions} of the joint system of
the KZ and dynamical
equations.

\begin{proof}

The proof uses the following identities:
\begin{gather*}
\frac 1{(t-z_i)(t-z_j)} = \frac1{z_i-z_j}\bigg(\frac1{t-z_i} - \frac1{t-zj}\bigg),
\\
\frac{\der}{\der t}\Phi_s(t,z,\la)
\equiv
p\la \Phi_s(t,z,\la) +\frac{p^s- 1}2 \Phi_s(t,z,\la) \sum_{j=1}^{2g+1} \frac 1{t-z_j} \pmod{p^s},
\\
\frac{\der}{\der t}\frac {\Phi_s(t,z,\la)}{t-z_i}
\equiv
\Phi_s(t,z,\la) \bigg(\frac{p\la}{t-z_i} +\frac{p^s- 3}2\frac1{(t-z_i)^2} +\frac{p^s- 1}2\sum_{j\ne i}
\frac{1}{(t-z_i)(t-z_j)}\bigg)
\end{gather*}
modulo $p^s$. For $j\ne i$, we have
\begin{align*}
\frac {\der }{\der z_i}\Phi_s(t,z,\la)\frac 1{t-z_j}
&= \Phi(t,z,\la) \frac {1-p^s}2\frac 1{(t-z_i)(t-z_j)}
\\
&
=\frac {1-p^s}2 \Phi_s(t,z,\la)\bigg(\frac 1{t-z_i} -\frac 1{t-z_j}\bigg).
\end{align*}
Take the coefficient of $t^{\ell p^s-1}$ in both sides of the equation.
As the result, we obtain modulo $p^s$,
\begin{gather*}
\frac {\der I_j^\ell}{\der z_i}(z,\la)
\equiv
\frac 12 \frac{I_i^\ell(z,\la)-I_j^\ell(z,\la)}{z_i-z_j}, \qquad i\ne j.
\end{gather*}
We have
\begin{align*}
\frac {\der }{\der z_i}\Phi_s(t,z,\la)\frac 1{t-z_i}
&
= \Phi_s(t,z,\la) \frac{3-p^s}2 \frac {1}{(t-z_i)^2}
\\
&
\equiv -\frac{\der}{\der t}\frac {\Phi_s(t,z,\la)}{t-z_i} +
 \Phi_s(t,z,\la) \bigg(\frac{p\la}{t-z_i} +\frac{p^s- 1}2\sum_{j\ne i}
\frac{1}{(t-z_i)(t-z_j)}\bigg).
\end{align*}
Take the coefficient of $t^{\ell p^s-1}$ in both sides of the equation.
As the result, we obtain modulo $p^s$,
\begin{gather*}
\frac {\der I_i^\ell}{\der z_i}(z,\la)
\equiv p\la I_i^\ell(z,\la) - \frac 12 \sum_{j\ne i}\frac{I_i^\ell(z,\la)-I_j^\ell(z,\la)}{z_i-z_J}.
\end{gather*}
Notice that $\frac{\der}{\der t}\frac {\Phi_s(t,z,\la)}{t-z_i}$ does not contribute to this result by Lemma
\ref{lem der}.

We also have modulo $p^s$,
\begin{align*}
\frac {\der }{\der \la}\Phi_s(t,z,\la)\frac 1{t-z_i}
&\equiv p\Phi_s(t,z,\la)\frac{t-z_i+z_i}{t-z_i} =
 p\Phi_s(t,z,\la) + pz_i \Phi_s(t,z,\la)\frac{1}{t-z_i}
\\
&\equiv
\frac 1\la
\frac{\der}{\der t}\Phi_s(t,z,\la)
+\frac 1\la \frac {1-p^s}{2} \Phi_s(t,z,\la)\! \sum_{j=1}^{2g+1} \!\frac 1{t-z_j}
+ p z_i \Phi_s(t,z,\la)\frac{1}{t-z_i} .
\end{align*}
Take the coefficient of $t^{\ell p^s-1}$ in both sides of the equation.
As the result, we obtain modulo $p^{s}$,
\begin{gather*}
\frac {\der I^\ell_i }{\der \la}(z,\la)
\equiv pz_i I_i^\ell(z,\la) + \frac 1{2\la} \sum_{j=1}^{2g+1} I_j^\ell(z,\la).
\end{gather*}
Notice that the coefficient of $t^{p^s-1}$ in $\frac 1\la
\frac{\der}{\der t}\Phi_s(t,z,\la) $ is zero modulo $p^{s}$ since $\la$ is a unit in $\Z_p$
by assumption.
The theorem is proved.
\end{proof}

\subsection[Properties of p\^s-hypergeometric solution]{Properties of $\boldsymbol{p^s}$-hypergeometric solutions}

\begin{Lemma}\label{lem numb}
Assume that $\la\in\Z_p^\times $ and
\begin{gather}
\label{ine}
p^s + 2g-1> s\frac{2p-2}{p-2}.
\end{gather}
Then a $p^s$-hypergeometric solution $I^\ell(z,\mu)$ is zero
unless $\ell =1,\dots,g$.
\end{Lemma}

\begin{proof} The degree of the polynomial $\Psi_s^o(t,z)$ with respect to $t$ equals
\begin{gather*}
(2g+1) \frac{p^s-1}2 -1 = (g+1)p^s -1 - \frac {p^s+2g+1}2.
\end{gather*}
The degree of the polynomial $E_s(p\la t)$ with respect to $t$ is not greater than
$d(s) \leq s\frac{p-1}{p-2}+1$. Hence the degree of $\Psi_s(t,z,\mu)$ is not greater than
\begin{gather*}
(g+1)p^s -1 - \frac {p^s+2g+1}2+ s\frac{p-1}{p-2} +1 = (g+1)p^s -1- \frac12\bigg(p^s+2g-1 - s\frac{2p-2}{p-2}\bigg).
\end{gather*}
If inequality \eqref{ine} holds, then the polynomial $\Psi_s(t,z,\la)$ does not have monomials of degree~$\ell p^s-1$
for $\ell > g$.
\end{proof}

For any $p^s$-hypergeometric solution $I^\ell(z,\la)$, consider its $\la$-independent term $I^\ell(z,0)=\big(I_1^\ell(z,0),
\dots,I_{2g+1}^\ell(z,0)\big)$.
This is a vector of polynomials in $z$ with
integer coefficients. It is a solution modulo $p^s$ of the KZ equations \eqref{KZ1p} and \eqref{KZ2p} with $\la=0$.
We have
\begin{gather*}
\sum_{j=1}^{2g+1} I_j^\ell(z,0)\equiv 0 \pmod{p^s}
\end{gather*}
 since this sum is the coefficient of $t^{\ell p^s-1}$ in
\begin{gather*}
\sum_{j=1}^{2g+1}\frac{\Phi_s^o(t,z)}{t-z_j} = \frac 2{p^s-1}\frac{\der \Phi^o_s}{\der t}(t,z).
\end{gather*}
The solution $I^\ell(z,\la)$ is a $\la$-deformation of the vector $I^\ell(z,0)$.

\begin{Theorem}[{\cite[Lemma 7.3]{V4}}]

Assume that $p^s>2g+1$. Consider the $\la$-independent terms $I^1(z,0), \dots, I^g(z,0)$ of the $p^s$-hypergeometric solutions
$I^1(z,\la), \dots, I^g(z,\la)$. Project them to~$\FF_p[z]^{2g+1}$. Then the projections are linearly independent over the ring
$\FF_p[z]$.

\end{Theorem}

\subsection{A generalization}
\label{sec more}

 Let
$r_1$, $r_2$ be relatively prime positive integers.
Denote $r=r_1/r_2$. Assume that $r>1/(p-1)$.
Change the variable $\la\to p^{r}\la$ in the KZ and dynamical equations \eqref{KZ1}, \eqref{KZ2},
\eqref{D}. Then the equations take the form
\begin{align}
\label{KZ1rp}
&\frac {\der I_j}{\der z_i}
=
\frac 12 \frac{I_i-I_j}{z_i-z_j}, \qquad i\ne j,
\\
\label{KZ2rp}
&\frac {\der I_i}{\der z_i}
=
p^{r}\la I_i - \frac 12 \sum_{j\ne i}\frac{I_i-I_j}{z_i-z_j},
\\
\label{Drp}
&\frac {\der I_i}{\der \la}
=
p^{r}z_i I_i + \frac 1{2\la} \sum_{j=1}^{2g+1} I_j.
\end{align}
For a positive integer $s$, define
\begin{gather*}
\Psi_{r,s}(t,z,\la) = E_{r,s}(p^{r}\la t)
\prod_{i=1}^{2g+1}(t-z_i)^{(p^s-1)/2}
\bigg(\frac{1}{t-z_1}, \dots, \frac{1}{t-z_{2g+1}}\bigg).
\end{gather*}
Consider the Taylor expansion
$
 \Psi_{r,s}(t,z,\la) = \sum_{d} c_d(z, \la) t^d$.
For any positive integer $\ell$, denote~$I^{\ell}(z,\la) = c_{\ell p^s-1}(z,\la)$.
All coordinates of this vector are polynomials in $z, \la$ with coefficients in $\Z_p\big[p^{1/r_2}\big]$.

\begin{Theorem}
\label{thm mod r}

Let $\ell$ be a positive integer. If $\la\in\Z_p\big[p^{1/r_2}\big]$, then
 $I^{\ell}(z,\la)$ is a solution modulo~$p^s$ of the KZ equations \eqref{KZ1rp}, \eqref{KZ2rp}.
If
$\la\in \big(\Z_p\big[p^{1/r_2}\big]\big)^\times$, then $I^{\ell}(z,\la)$
 is a solution modulo $p^{s}$ of the dynamical equations \eqref{Drp}.

\end{Theorem}

The proof is the same as the proof of Theorem \ref{thm mod}. Theorem \ref{thm mod}
is a special case of Theorem~\ref{thm mod r} for $r=1$.

Notice that for degree reasons, Theorem \ref{thm mod r} gives for every $s$ only finitely many solutions~$I^\ell(z,\la)$.

 If $r=1/(p-1)$ and $s$ is a positive integer, we may define
\begin{gather*}
\Psi_{1/(p-1),s}(t,z,\la) = {\rm e}^{p^{1/(p-1)}\la t}
\prod_{i=1}^{2g+1}(t-z_i)^{(p^s-1)/2}
\bigg(\frac{1}{t-z_1}, \dots, \frac{1}{t-z_{2g+1}}\bigg)
\end{gather*}
and then expand this vector into a power series in $t$:\
$ \Psi_{1/(p-1),s}(t,z,\la) = \sum_{d} c_d(z, \la) t^d$.
For any positive integer $\ell$, denote
$I^{\ell}(z,\la) = c_{\ell p^s-1}(z,\la)$.
All coordinates of this vector are polynomials in~$z$,~$\la$ with coefficients in $\Z_p\big[p^{1/(p-1)}\big]$.

\begin{Theorem}
\label{thm p-1}

Let $\ell$ be a positive integer. If $\la\in\Z_p\big[p^{1/(p-1)}\big]$, then
 $I^{\ell}(z,\la)$ is a solution modulo~ $p^s$ of the KZ equations \eqref{KZ1rp}, \eqref{KZ2rp} with $r=1/(p-1)$
If
$\la\in \big(\Z_p\big[p^{1/(p-1)}\big]\big)^\times$, then~$I^{\ell}(z,\la)$
 is a~solution modulo $p^{s}$ of the dynamical equations \eqref{Drp} with $r=1/(p-1)$.
 \end{Theorem}

 Notice that this theorem gives infinitely many solutions $I^{\ell}(z,\la)$.

\begin{Remark}
Another possibility to extend the construction of polynomial solutions is to replace the ring
$\Z_p[p^{r}]$ by another $p$-adic ring, e.g., $\Z_p[\zeta]$, where $\zeta$
 is a $p^m$-th root of 1.
\end{Remark}

\subsection{Exposition of material}
In Section \ref{sec 2}, we describe the hypergeometric solutions of the joint system of the
differential KZ and dynamical
equations associated with $\slt$ and explain their reduction to polynomial solutions modulo $p^s$. In Section \ref{4}, we
briefly comment on
 how the results of Section \ref{sec 2}
 are extended to the joint system of the differential KZ and dynamical equations associated with arbitrary simple Lie algebras.
 In Appendix \ref{sec app}, we consider an example and explain how to construct the
 polynomial solutions modulo $p^s$ of
 qKZ difference equations.

\section[The sl\_2 differential KZ and dynamical equations]{The $\boldsymbol{\mathfrak{sl}_2}$ differential KZ and dynamical equations}
\label{sec 2}

\subsection{Equations}
{\samepage Let $e$, $f$, $h$ be the
standard basis of the complex Lie algebra $\slt$ with relations $[e,f]=h$, ${[h,e]=2e}$, $[h,f]=-2f$. Denote
\[
\Omega = e \otimes f + f \otimes e +
 \frac{1}{2} h \otimes h \in \slt \ox\slt,
\]
the Casimir element.}

Given $n$, for any $x\in \slt$, let $x^{(i)} \in U(\slt)^{\ox n}$
 be the element equal to $x$ in the $i$-th factor and to 1 in other factors.
Similarly, for $1\leq i<j\leq n$, let $\Om^{(i,j)} \in U(\slt)^{\ox n}$
 be the element equal to~$\Omega$ in the $i$-th and $j$-th factors and to~1 in other factors.

Let $z_1,\dots, z_n \in\C$ be distinct and $\la \in\C^\times$.
 For $i=1,\dots,n$, introduce the Gaudin Hamiltonians and the dynamical Hamiltonian
 by the formulas
\begin{gather*}
 H_i(z_1,\dots,z_{n}, \la)
 =
 \frac \la 2 h^{(i)} +\sum_{j\ne i}\frac{\Om^{(i,j)}}{z_i-z_j} \in (U(\slt))^{\ox n},
 \\
 D(z_1,\dots,z_{n}, \la)
 =
 \sum_{i=1}^n \frac{z_i}2h^{(i)} + \sum_{i,j=1}^n \frac{f^{(i)}{\rm e}^{(j)}}\la.
\end{gather*}

Let $ \ox^n_{i=1}V_i$ be a tensor product of $\slt$-modules and $\ka\in\C^\times$.
The system of differential equations on a $ \ox^n_{i=1}V_i$-valued function
$I(z_1, \dots, z_n,\la)$,
\begin{gather}\label{kz}
 \frac{\partial I}{\partial z_i}
 =
 \frac{1}{\kappa} H_i(z_1,\dots,z_{n}, \la) I,
\qquad
 i = 1, \dots , n,
\\
\label{de}
\frac{\partial I}{\partial \la}
 =
 \frac{1}{\kappa} D(z_1,\dots,z_{n}, \la) I,
\end{gather}
is called the system of KZ and dynamical equations.
The system depends on the parameter $\ka$.

\subsection[slt-modules]{$\boldsymbol{\slt}$-modules}
For a nonnegative integer $i$, denote by $L_i$ the $(i+1)$-dimensional module with a
basis $v_i, fv_i,\dots,\allowbreak f^iv_i$ and action
\begin{gather*}
 f\cdot f^kv_i
 =
 f^{k+1}v_i \qquad \text{for}\quad k=0,\dots, i-1,
\\
h\cdot f^kv_i
= (i-2k)f^kv_i \qquad \text{for}\quad
k=0,\dots,i,
\\
e\cdot f^kv_i = k(i-k+1)f^{k-1}v_i\qquad\text{for} \quad k=1,\dots,i,
\end{gather*}
$f\cdot f^iv_i=0$, $e\cdot v_i=0$.

For $\vec m = (m_1,\dots,m_n) \in \Z^n_{\geq 0}$, denote
$|\vec m| = m_1 + \dots + m_n$ and
$L^{\otimes \vec m} = L_{m_1} \otimes \dots \otimes L_{m_n}$.
 For~$J=(j_1,\dots,j_n) \in \mathbb{Z}_{\geq 0}^n$, with $j_s\leq m_s$ for $s=1,\dots,n$, the vectors
\[
f_Jv := f^{j_1}v_{m_1}\otimes \dots \otimes f^{j_n}v_{m_n}
\]
form a basis of
$L^{\otimes \vec m}$. We have
\begin{gather*}
f\cdot f_Jv = \sum_{s=1}^n f_{J+1_s}v,
\qquad
h\cdot f_Jv = ( |m|-2|J|) f_Jv,
\\
e\cdot f_Jv = \sum_{s=1}^n j_s (m_s-j_s+1) f_{J-1_s}v ,
\end{gather*}
where $1_s=(0,\dots, 0, 1,0,\dots,0)$ with 1 staying at the $s$-th place.

For $w \in \Z$, introduce the weight subspace $L^{\otimes \vec m}[w] = \{ v \in L^{\otimes \vec m} \mid h.v = w v \}$. We have the weight decomposition
$L^{\ox \vec m} = \oplus_{k=0}^{|m|}L^{\ox \vec m}[|\vec m|-2k]$.
Denote
\begin{gather*}
\Ik =\{ J\in \Z^n_{\geq 0}\mid |J|=k,\, j_s\leq m_s,\, s=1,\dots,n\}.
\end{gather*}
The vectors $(f_Jv)_{J\in\Ik}$ form a basis of $L^{\otimes \vec m}[|\vec m|-2k]$.

\subsection[Solutions over C]{Solutions over $\boldsymbol{\mathbb C}$}
 \label{Sol in C}

Given $k, n \in \Z_{>0}$, $\vec m = ( m_1, \dots , m_n) \in \Z_{> 0}^n$, $\ka\in\C^\times$,
denote $t=(t_1,\dots,t_k)$, $z=(z_1,\dots,z_n)$. Define the {\it master function}
\begin{align*}
\Phi(t,z,\la) : ={}& \Phi (t_1, \dots , t_k, z_1, \dots , z_n, \la) ={\rm e}^{\la\sum_{l=1}^nz_l/2\ka -\la\sum_{i=1}^k t_i/\ka }
\\
 & {}\times
\prod_{i<j} (z_i-z_j)^{m_im_j/2\ka}
 \prod_{1 \leq i \leq j \leq k} (t_i-t_j)^{2/\ka}
 \prod_{l=1}^{n} \prod_{i=1}^{k} (t_i-z_l)^{-m_l/\ka}.
\end{align*}
For any function or differential form $F(t_1, \dots , t_k)$, denote
\begin{gather*}
\Sym_t [F(t_1, \dots , t_k)] = \sum_{\sigma \in S_k}
F(t_{\sigma_1}, \dots , t_{\sigma_k}) ,
\\
\text{Alt}_t [F(t_1, \dots , t_k)] = \sum_{\sigma \in S_k}
(-1)^{|\sigma|}F(t_{\sigma_1}, \dots , t_{\sigma_k}) .
\end{gather*}
For $J=(j_1,\dots,j_n)\in \Ik$, define the {\it weight function}
\begin{gather*}
W_J(t,z) = \frac{1}{j_1! \dots j_n!}
 \Sym_t \left[ \prod_{s=1}^{n}
 \prod_{i=1}^{j_s}
 \frac{1}{t_{ j_1 + \dots + j_{s-1}+i} - z_s}
 \right] .
\end{gather*}
For example,
\begin{gather*}
W_{(1,0,\dots,0)} = \frac{1}{t_1 - z_1} ,
\qquad
W_{(2,0,\dots,0)} = \frac{1}{t_1 - z_1} \frac{1}{t_2 - z_1} ,
\\
W_{(1,1,0,\dots,0)} = \frac{1}{t_1-z_1} \frac{1}{t_2-z_2}
+ \frac{1}{t_2-z_1} \frac{1}{t_1-z_2} .
\end{gather*}
The function
\[
W(t,z)=\sum_{J\in \Ik}W_J(t,z) f_Jv
\]
is the $L^{\otimes \vec m}[ |\vec m| -2k]$-valued {\it vector weight function}.

Consider the $L^{\otimes \vec m}[ |\vec m| -2k]$-valued
function
\begin{gather}
\label{intrep}
I^{(\dl)}(z_1, \dots , z_n, \la) = \int_{\dl(z,\la)} \Phi(t,z, \la)
 W(t,z) {\rm d}t_1 \wedge \dots \wedge {\rm d}t_k ,
\end{gather}
where $ \dl(z,\la)$ in $\{(z,\la)\} \times \C^k_t$ is a horizontal family of
$k$-dimensional cycles of the twisted homology defined by the multivalued
function $\Phi(t,z, \la)$, see, for example,~\cite{V1,V2}.
The cycles~$\dl(z,\la)$ are multi-dimensional analogs of Pochhammer double loops.

\begin{Theorem}[\cite {FMTV,SV2}]
\label{thm s}
The function $I^{(\dl)} (z,\la)$ is a solution of the KZ and dynamical equations~\eqref{kz} and~\eqref{de}.
\end{Theorem}

The solutions in \eqref{intrep} are called the hypergeometric solutions.

The equations \eqref{KZ1}, \eqref{KZ2}, \eqref{D} in the Introduction and their solutions \eqref{int rep}
are identified with equations \eqref{kz}, \eqref{de} for the weight subspace
$L_1^{\ox (2g+1)}[n-2]$ and their hypergeometric solutions \eqref{intrep} up to a gauge transformation.

In Section \ref{sec proof C}, we sketch the proof of Theorem \ref{thm s}
following \cite{FMTV, SV2}.
The intermediate statement in this proof will be used later when constructing solutions
modulo $p^s$ of the KZ and dynamical equations. The proof is based on the following cohomological relations.

\subsection{Identities for differential forms}
\label{sec proof C}

It is convenient to reformulate the definition of the hypergeometric integrals
\eqref{intrep}. Given ${k, n \in \Z_{>0}}$ and a multi-index $J = (j_1, \dots , j_n)$ with $|J| \leq k$,
denote
\begin{align*}
a_J = a_{J,1} \wedge a_{J,2}\wedge \dots \wedge a_{J,|J|}
: ={}&
\frac{{\rm d}(t_1 - z_1)}{t_1 -z_1} \wedge
\dots
\wedge \frac{{\rm d}(t_{j_1} - z_1)}{t_{j_1} - z_1} \wedge
\frac{{\rm d}(t_{j_1+1} - z_2)}{t_{j_1+1} - z_2} \wedge \cdots
\\
&\wedge
\frac{{\rm d}(t_{j_1+ \dots + j_{n-1}+1} - z_n)}{t_{j_1+ \dots + j_{n-1}+1} - z_n}
\wedge \dots \wedge
\frac{{\rm d}(t_{j_1+ \dots + j_{n}} - z_n)}{t_{j_1+ \dots + j_{n}} - z_n}.
\end{align*}
Here $a_{J, \ell}$ is the $\ell$-th factor of the product in the right-hand side.

Denote
\begin{gather*}
b_J = b_{J,1} \wedge b_{J,2}\wedge \dots \wedge b_{J,|J|}
: =
\frac{{\rm d}t_1}{t_1 -z_1} \wedge
\dots
\wedge \frac{{\rm d}t_{j_1}}{t_{j_1} - z_1} \wedge
\frac{{\rm d}t_{j_1+1} }{t_{j_1+1} - z_2} \wedge \cdots
\\
\hphantom{b_J = b_{J,1} \wedge b_{J,2}\wedge \dots \wedge b_{J,|J|}
: =}{} \wedge
\frac{{\rm d}t_{j_1+ \dots + j_{n-1}+1}}{t_{j_1+ \dots + j_{n-1}+1} - z_n}
\wedge \dots \wedge
\frac{{\rm d}t_{j_1+ \dots + j_{n}} }{t_{j_1+ \dots + j_{n}} - z_n},
\\
c_J = \sum_{l=1}^{|J|} (-1)^{l+1} b_{J,1} \wedge b_{J,2}\wedge \cdots \wedge\widehat{b_{J,l}}
\wedge \cdots\wedge b_{J,k}.
\end{gather*}
Define
\[
\alpha_J = \frac{1}{j_1 ! \cdots j_n !}\Alt_{t_1,\dots,t_k} [a_J], \qquad
\beta_J = \frac{1}{j_1 ! \cdots j_n !}
\Alt_{t_1,\dots,t_k} [c_J].
\]

\begin{Remark}
Recall that we have $k$ variables $t_1,t_2,\dots,t_k$.
The differential form $a_J$ is of degree~${|J|=j_1+\dots+j_n \leq k}$
and depends on the variables $t_1, t_2,\dots, t_{j_1+\dots+j_n}$ only. While, the differential form $\alpha_J$ is of degree
$j_1+\dots+j_n$ and depends on all the variables $t_1,t_2, \dots,t_k$.
\end{Remark}

If $|J| = k$, then for any fixed $z\in \C^n$, we have the identity
\[ \alpha_J = W_J(t, z) {\rm d}t_1 \wedge \dots \wedge {\rm d}t_k \]
on $\{z\} \times \C^k$.
Define
\[
\alpha = \sum_{|J|=k} \alpha_J f_Jv ,
\qquad
\beta = \sum_{|J|=k} \beta_aJ f_Jv .
\]

\begin{Example}
For $k = n = 2$, we have
\begin{gather*}
\alpha_{(2,0)}
= \frac{{\rm d}(t_1 - z_1)}{t_1 -z_1} \wedge \frac{{\rm d}(t_2 - z_1)}{t_2 -z_1} ,
\qquad
\beta_{(2,0)}
= \frac{{\rm d} t_2}{t_2 -z_1}
- \frac{{\rm d}t_1 }{t_1 -z_1} .
\end{gather*}
\end{Example}

The hypergeometric integrals \eqref{intrep} can be defined in terms of
the differential forms $\alpha_J$:
\[
I^{(\delta)} (z_1, \dots , z_n,\lambda) =
\int_{\delta (z,\lambda)}\Phi \alpha = \sum_{J\in \Ik} \bigg(\int_{\delta (z,\lambda)}\Phi \alpha_J \bigg)f_J v .
\]

\begin{Theorem} [\cite{FMTV,SV3}]\label{ss2}\quad
\begin{itemize}\itemsep=0pt
\item[$(i)$]
We have the following algebraic identity for differential forms in $t$, $z$ depending on parameter $\la$:
\begin{gather}
\label{id1}
\kappa {\rm d}_{t,z}\big(\Phi(t,z,\lambda)\alpha\big) = \sum_{l=1}^n H_i(z,\lambda){\rm d}z_i\wedge \Phi(t,z,\lambda)\alpha,
\end{gather}
where ${\rm d}_{t,z}$ denotes the differential with respect to variables $t$, $z$.

\item[$(ii)$]
For any fixed $z$, $\lambda$, we have the following algebraic identity for differential forms in $t$
depending on parameters $z$, $\lambda$:
\begin{gather}
\label{id2}
\kappa \frac{\der}{\der\lambda} \big(\Phi(t,z,\lambda)\alpha\big)
= D(t,z,\lambda) \Phi(t,z,\lambda) \alpha +
 \frac1{\lambda}
{\rm d}_t\big(\Phi(t,z,\lambda) \beta\big),
\end{gather}
where ${\rm d}_t$ denotes the differential with respect to variables $t$.
\end{itemize}
\end{Theorem}

The assumptions in part $(ii)$ mean that all differentials ${\rm d}z_1,\dots,{\rm d}z_n$ appearing in $\alpha$ must be put to zero to obtain identity \eqref{id2}.

\begin{proof}
Identity \eqref{id1} follows from \cite[Theorem 7.5.2$''$]{SV3}
and \cite[Theorem 3.1]{FMTV}. Identity \eqref{id2} follows from \cite[Theorem 3.2]{FMTV}.
\end{proof}

Integrating both sides of identities \eqref{id1} and \eqref{id2} over the cycle $\delta(z,\lambda)$
we conclude that the integral
$\int_{\delta(z,\lambda)}\Phi(t,z,\lambda)\alpha$ satisfies the KZ and dynamical equations.

\subsection[Solutions modulo p\^s]{Solutions modulo $\boldsymbol{p^s}$}
 \label{Sol in F}

Given $k, n \in \Z_{>0}$, $\vec m = ( m_1, \dots , m_n) \in \Z_{>0}^n$, $\ka\in\mathbb Q^\times$,
let $p>2$ be a prime number such that~$p$ does not divide the numerator of $\ka$.
Change the variable $\la\to p\la$ in the KZ and dynamical equations \eqref{kz}, \eqref{de}.
Then the equations take the form
\begin{gather}
\label{kzp}
 \frac{\partial I}{\partial z_i}
 =
\frac 1\ka \bigg( p \frac {\la}{2} h^{(i)} + \sum_{j\ne i}\frac{\Om^{(i,j)}}{z_i-z_j}\bigg) I,
\qquad
 i = 1, \dots , n,
\\
\label{dep}
\frac{\partial I}{\partial \la}
 =
\frac1\ka\bigg(p \sum_{i=1}^n \frac{z_i}{2} h^{(i)}
 + \sum_{i,j=1}^n \frac{f^{(i)}{\rm e}^{(j)}}\la \bigg) I.
\end{gather}

 Choose positive integers $M_l$ for $l=1,\dots,n$,
$M_{i,j}$ for $1\le i<j\leq n$, and $M^0$, such that
\[
M_s\equiv -\frac{m_s}\ka,
\qquad
M_{i,j}\equiv \frac{m_im_j}{2\ka},
\qquad
M^0\equiv \frac2\ka
\qquad \pmod{p^s}.
\]
Define the {\it master polynomial}
\begin{align*}
\Phi_s(t,z,\la)
={}&
\prod_{l=1}^n E_s\bigg(p\frac{z_l\la}{2\ka}\bigg) \prod_{i=1}^n E_s\bigg(-p\frac{t_i\la}{\ka}\bigg)
\times
\\
 &{}\times \prod_{1\leq i<j\leq n} (z_i-z_j)^{M_{i,j}}
 \prod_{1 \leq i \leq j \leq k} (t_i-t_j)^{M^0}
 \prod_{s=1}^{n} \prod_{i=1}^{k} (t_i-z_s)^{M_s}.
\end{align*}
Consider the $L^{\otimes \vec m}[ |\vec m| -2k]$-valued
function
\begin{gather*}
\Psi_s(t,z,\la) = \sum_{J\in \Ik}\Phi_s(t,z,\la)W_J(t,z) f_Jv.
\end{gather*}
This is a polynomial in $t$, $z$, $\la$ with coefficients in $\Z_p$.
Consider the Taylor expansion
\begin{gather*}
\Psi_s(t,z,\la) = \sum_{{\rm d}_1,\dots,{\rm d}_k} c_{{\rm d}_1,\dots,{\rm d}_k}(z, \la) t_1^{{\rm d}_1}\dots t_k^{{\rm d}_k}.
\end{gather*}
For any vector $\ell=(\ell_1,\dots,\ell_k)$ with positive integer coordinates, denote
\begin{gather*}
I^{\ell}(z,\la) = c_{\ell_1 p^s-1,\dots, \ell_k p^s-1}(z,\la).
\end{gather*}
All coordinates of this vector are polynomials in $z$, $\la$ with coefficients in $\Z_p$.

\begin{Theorem}
\label{thm 2.4}
Let $\ell=(\ell_1,\dots,\ell_k)$ be a vector with positive integer coordinates.
 If $\la\in\Z_p$, then~$I^{\ell}(z,\la)$ is a solution modulo $p^s$ of the KZ equations \eqref{kzp}.
 If $\la\in \Z_p^\times$, then $I^{\ell}(z,\la)$
 is a~solution modulo $p^{s}$ of the dynamical equations \eqref{dep}.
\end{Theorem}

We call such solutions the {\it $p^s$-hypergeometric solutions} of the joint system of
the KZ and dynamical equations.

\begin{proof}
The polynomial $\Psi_s(t,z,\la)$ satisfies identities \eqref{id1} and \eqref{id2} modulo $p^s$. Taking the coefficient
of $t_1^{\ell_1 p^s-1}\dots t_k^{\ell_k p^s-1}$
in these identities kills the differentials with respect to $t$ by Lemma~\ref{lem der}. This proves the theorem.
\end{proof}

\begin{Remark}
We observed in Section \ref{sec more} that Theorem \ref{thm mod} can be generalized to
 Theorem \ref{thm mod r} by replacing $p$ with $p^{r}$. Theorem \ref{thm 2.4}
is generalized in the same way. We leave this exercise to readers.
\end{Remark}

\subsection{Equations for other Lie algebras}
\label{4}

The KZ and dynamical equations are
defined for any simple Lie algebra $\mathfrak{g}$ or more generally for any Kac--Moody algebra, see, for example, \cite{FMTV, SV2}.
Similarly to what is done in Section~\ref{Sol in F}, we can construct polynomial solutions modulo $p^s$
of those KZ and dynamical equations.

 The construction of the polynomial solutions modulo $p^s$ in the $\slt$ case is based on
the algebraic identities for differential forms \eqref{id1}, \eqref{id2}.
For an arbitrary Kac--Moody algebra, these
 algebraic identities were developed in \cite{FMTV, SV2}.

\appendix
\section[Solutions modulo p\^s of the qKZ equations]{Solutions modulo $\boldsymbol{p^s}$ of the qKZ equations}
\label{sec app}

In Sections \ref{sec 1} and \ref{sec 2}, we constructed solutions modulo $p^s$ of the differential KZ and
dynamical equations by first
$p^s$-approximating the integrand of the hypergeometric solutions and then taking
the coefficients of the monomials
$t^{\ell p^s-1}$ in the Taylor
expansion of the approximated integrand. In this appendix, we show that
the same idea can be applied to the qKZ difference equations, but instead of considering
the Taylor expansion of the approximated integrand and then taking the coefficients of the monomials $t^{\ell p^s-1}$ we
first take the expansion of the $p^s$-approximated integrand into a sum of Pochhammer polynomials
$[t]_m$ and then take the coefficients of the Pochhammer polynomials with indices $m=\ell p^s-1$.
See this approach in \cite{MuV} for the qKZ equations with no exponential term.
On the hypergeometric solutions of the qKZ difference equations see, for example, \cite{TV1,TV3}.

In this paper, we consider only a baby example of the qKZ equations which illustrates these constructions.

\subsection{Baby qKZ equation}

Let $f(t)$ be a meromorphic function and $a\in\C$. The sum
\begin{gather*}
\int_{(a)} f(t) {\rm d}_1t :=\sum_{n\in\Z}\text{Res}_{t=a+n}f(t),
\end{gather*}
if defined, is called a Jackson integral. If
$f(t) = g(t+1)-g(t)$ is a discrete differential, then the Jackson integral equals zero.

Consider the master function
\begin{gather}
\label{qkz ma}
\Phi (t,z,\la) = {\rm e}^{\la t} \frac{\Gamma(t-z)^2}{\Gamma\left(t-z + \frac12\right)\Gamma
\left(t-z - \frac12\right)}
\end{gather}
and the Jackson integral
\begin{gather*}
I(z,\la) = \int_{(z)} \Phi (t,z,\la) {\rm d}t.
\end{gather*}

\begin{Lemma}
The function $I(z,\la)$ satisfies the difference equation
\begin{gather}
\label{qkz}
{\rm e}^\la I(z-1,\la) = I(z,\la) .
\end{gather}
\end{Lemma}

Equation \eqref{qkz} is our baby qKZ equation and the function
$I(z,\la)$ is its hypergeometric solution. More general qKZ equations and their hypergeometric solutions can be found,
 for example, in \cite{EFK,FR,TV1,TV3}.

\begin{proof}
We have
\begin{gather*}
\Phi (t+1,z,\la) = \frac{{\rm e}^{\la}(t-z)^2}{\left(t-z + \frac12\right)\left(t-z - \frac12\right)}\Phi(t,z,\la),
\\
\Phi(t,z-1, \la)=\frac{(t-z)^2}{\left(t-z + \frac12\right)\left(t-z - \frac12\right)} \Phi(t,z,\la),
\\
\int_{(z)} \Phi (t+1,z,\la)d_1t = \int_{(z)}\Phi (t,z,\la)d_1t.
\end{gather*}
These relations imply \eqref{qkz}.
\end{proof}

\subsection{Pochhammer polynomials}
Let $m$ be a positive integer. Define the Pochhammer polynomial
\begin{gather*}
(t)_m=\prod_{i=1}^{m}(t-i+1).
\end{gather*}
We have
\begin{gather}
\label{dif}
(t+1)_m = (t)_m\frac{t + 1}{t-m+1},
\qquad
(t+1)_m -(t)_m = m (t)_{m-1}.
\end{gather}

\subsection[Polynomial p\^s-approximations]{Polynomial $\boldsymbol{p^s}$-approximations}

Let $p$ be an odd prime.
 Let $r_1$, $r_2$ be relatively prime positive integers.
Denote $r=r_1/r_2$. Assume that $r>1/(p-1)$.
Change the variable $\la\mapsto p^{r}\la$. Then equation
\eqref{qkz} takes the form:
\begin{gather}
\label{qkzp}
{\rm e}^{p^{r}\la} I(z-1,\la) = I(z,\la) .
\end{gather}
For any positive integer $s$, define the master polynomial
\begin{gather}
\label{ma p}
\Phi_{r,s} (t,z,\la) = E_{r,s}(p^{r}\la t) (t-z-1)_{\frac{p^s-1}2} (t-z-1)_{\frac{p^s+1}2}.
\end{gather}
This is a polynomial in $t$, $z$, $\la$ with coefficients in $\Z_p\big[p^{1/r_2}\big]$.
Expand it into a sum of Pochhammer polynomials in the variable $t$,
\begin{gather*}
\Phi_{r,s} (t,z,\la) = \sum_d c_d(z,\la) (t)_d.
\end{gather*}
Denote $I_{r,s}(z,\la) = c_{p^s-1}(z,\la)$.

\begin{Theorem}
\label{thm last}
The polynomial $I_{r,s}(z,\la)$ is a solution modulo $p^s$ of the qKZ difference equation~\eqref{qkzp}.

\end{Theorem}

\begin{proof}
Let $F(t)=\sum_d a_d(t)_d$ be a polynomial. Denote
$\text{Coef}(F) = a_{p^s-1}$.
We have
\begin{gather}
\Phi_{r,s} (t+1,z,\la) = \frac{E_{r,s}(p^{r}\la) (t-z)^2}
{\left(t-z-\frac{p^s-1}2\right)\Big(t-z -\frac{p^s+1}2\Big)} \Phi_{r,s}(t,z,\la),\nonumber\\
\Phi_{r,s}(t,z-1, \la) =\frac{(t-z)^2}
{\left(t-z-\frac{p^s-1}2\right)\left(t-z -\frac{p^s+1}2\right)}
 \Phi_{r,s}(t,z,\la),\nonumber\\
\label{eq3p}
\text{Coef}(\Phi_{r,s} (t,z,\la))\equiv \text{Coef}(\Phi_{r,s} (t+1,z,\la)) \pmod{p^s},
\end{gather}
where the congruence \eqref{eq3p} follows from \eqref{dif}. Hence
\begin{gather*}
E_{r,s}(p^{r}\la)I_{r,s}(z-1,\la) \equiv I_{r,s}(z,\la) \pmod{p^s},
\end{gather*}
and $I_{r,s}(z,\la)$ is a solution modulo $p^s$ of the qKZ equation \eqref{qkzp}.
\end{proof}

More general qKZ equations are given by multidimensional Jackson integrals whose integrand is a product of exponential factors and
ratios of gamma functions like in \eqref{qkz ma}. To construct polynomial solutions modulo $p^s$ of the qKZ equations,
 we replace the exponential factors in the integrand by the product of the corresponding functions $E_{r,s}(t)$,
 replace the ratios of gamma functions by
 the corresponding Pochhammer polynomial as in~\eqref{ma p}, expand the result into Pochhammer polynomials
 and take the suitable coefficients of that expansion like in Theorem~\ref{thm last}.

\subsection*{Acknowledgements}

Pavel Etingof was supported in part by NSF grant DMS-1916120,
Alexander Varchenko was supported in part by NSF grant DMS-1954266.

\pdfbookmark[1]{References}{ref}
\LastPageEnding

\end{document}